\documentclass[12pt,reqno,a4paper]{amsart}

\usepackage{amsfonts,amsthm,amsmath,amscd,amssymb}
\usepackage{epsfig}
\usepackage{graphicx}

\newcounter{main}

\numberwithin{equation}{section}

\newtheorem{theorem}{Theorem}[section]
\newtheorem{proposition}[theorem]{Proposition}
\newtheorem{lemma}[theorem]{Lemma}

\newtheorem{remark}[theorem]{Remark}
\newtheorem{definition}[theorem]{Definition}
\newtheorem{maintheorem}{Theorem}

\newcommand{\tra}{\overline}
\newcommand{\muo}{\tra\mu}

\newcommand{\Rr}{{\mathbb{R}}}
\newcommand{\Nn}{{\mathbb{N}}}

\newcommand{\LL}{{\mathcal L}}
\newcommand{\NN}{{\mathcal N}}
\newcommand{\EE}{{\mathcal E}}
\newcommand{\RR}{{\mathcal R}}
\newcommand{\VV}{{\mathcal V}}
\newcommand{\En}{{\mathcal E}}

\newcommand{\pde}[2]{\frac{\partial #1}{\partial #2}}

\def \dim{\operatorname{dim}}
\def \Sing {\operatorname{Crit}}

\def \LE{\operatorname{LE}}

\begin{document}

\title[Generic dynamics of 4-dim Hamiltonians]
{Generic dynamics of $4$-dimensional $C^2$ Hamiltonian systems}

\author[M. Bessa]{M\'{a}rio Bessa}
\address{Centro de Matem\'atica da Universidade do Porto, Rua do Campo Alegre, 687, 4169-007 Porto, Portugal}
\email{bessa@impa.br}

\author[J. Lopes Dias]{Jo\~{a}o Lopes Dias}
\address{Departamento de Matem\'atica, ISEG, 
Universidade T\'ecnica de Lisboa,
Rua do Quelhas 6, 1200-781 Lisboa, Portugal}
\email{jldias@iseg.utl.pt}

\date{December 17, 2007}

\begin{abstract}
We study the dynamical behaviour of Hamiltonian flows defined on $4$-dimensional compact symplectic manifolds.
We find the existence of a $C^2$-residual set of Hamiltonians for which there is an open mod 0 dense set of regular energy surfaces each being either Anosov or having zero Lyapunov exponents almost everywhere.
This is in the spirit of the Bochi-Ma\~n\'e dichotomy for area-preserving diffeomorphisms on compact surfaces~\cite{B} and its continuous-time version for $3$-dimensional volume-preserving flows~\cite{Be}.
\end{abstract}

\maketitle

\begin{section}{Introduction and statement of the results}

The computation of Lyapunov exponents is one of the main problems in the modern theory of dynamical systems. 
They give us fundamental information on the asymptotic exponential behaviour of the linearized system. 
It is therefore important to understand these objects in order to study the time evolution of orbits. 
In particular, Pesin's theory deals with non-vanishing Lyapunov exponents systems (non-uniformly hyperbolic). This setting jointly with a $C^{\alpha}$ regularity, $\alpha>0$, of the tangent map allows us to derive a very complete geometric picture of the dynamics (stable/unstable invariant manifolds). On the other hand, if we aim at understanding both local and global dynamics, the presence of zero Lyapunov exponents creates lots of obstacles.
An example is the case of conservative systems: using enough differentiability, the celebrated KAM theory guarantees persistence of invariant quasiperiodic motion on tori yielding zero Lyapunov exponents.

In this paper we study the dependence of the Lyapunov exponents on the dynamics of Hamiltonian flows.
Despite the fact that the theory of Hamiltonian systems ask, in general, for more refined topologies, here we work in the framework of the $C^1$ topology of the Hamiltonian vector field.
Our motivation comes from a recent result of Bochi~\cite{B} for area-preserving diffeomorphisms on compact surfaces, followed by its continuous time counterpart~\cite{Be} for volume-preserving flows on compact $3$-manifolds.
We point out that these results are based on the outlined approach of Ma\~{n}\'{e}~\cite{M1,M2}. 
Furthermore, Bochi and Viana (see~\cite{BV2}) generalized the result in~\cite{B} and proved also a version for linear cocycles and symplectomorphisms in any finite dimension. For a survey of the theory see~\cite{BV3} and references therein. 

Here we prove that zero Lyapunov exponents for $4$-dimensional Hamiltonian systems are very common, at least for a $C^2$-residual subset.
This picture changes radically for the $C^\infty$ topology, the setting of most Hamiltonian systems coming from applications. In this case Markus and Meyer showed that there exists a residual of $C^\infty$ Hamiltonians neither integrable nor ergodic \cite{Markus}.

\medskip

Let $(M,\omega)$ be a compact symplectic manifold. We will be interested on the Hamiltonian dynamics of real-valued $C^s$, $2\leq s\leq+\infty$, functions on $M$ that are constant on each comnected component of the boundary $\partial M$. These functions are referred to as Hamiltonians on $M$ and their set will be denoted by $C^s(M,\Rr)$ which we endow with the $C^2$-topology. Moreover, we include in this definition the case of $M$ without boundary $\partial M=\emptyset$.
We assume $M$ and $\partial M$ (when it exists) to be both smooth.

Given a Hamiltonian $H$, any scalar $e\in H(M)\subset\Rr$ is called an energy of $H$ and $H^{-1}(e)=\{x\in M\colon H(x)=e\}$ the corresponding invariant energy level set. 
It is regular if it does not contain critical points.
An energy surface $\En$ is a connected component of $H^{-1}(e)$. 
Notice that connected components of $\partial M\not=\emptyset$ correspond to energy surfaces.

The volume form $\omega^d$ gives a measure $\mu$ on $M$ that is preserved by the Hamiltonian flow.
Recall that for a $C^2$-generic Hamiltonian all but finitely many points are regular (hence a full $\mu$-measure set), since Morse functions are $C^2$-open and dense.
On each regular energy surface $\EE\subset M$ there is a natural finite invariant volume measure $\mu_\EE$ (see section \ref{subsection:basic notions}).

\begin{maintheorem}\label{teorema22}
Let $(M,\omega)$ be a $4$-dim compact symplectic manifold. 
For a $C^2$-generic Hamiltonian $H\in C^2(M,\Rr)$, the union of the regular energy surfaces $\EE$ that are either Anosov or have zero Lyapunov exponents $\mu_\EE$-a.e. for the Hamiltonian flow, forms an open $\mu$-mod $0$ and dense subset of $M$.
\end{maintheorem}

A regular energy surface is Anosov if it is uniformly hyperbolic for the Hamiltonian flow (cf. section \ref{section:hyperb}).
Geodesic flows on negative curvature surfaces are well-known systems yielding Anosov energy levels. An example of a mechanical system which is Anosov on each positive energy level was obtained by Hunt and MacKay~\cite{Hunt}.

We prove another dichotomy result for the transversal linear Poincar\'{e} flow on the tangent bundle (see section \ref{subsection:transversal linear Poincare flow}). This projected tangent flow can present a weaker form of hyperbolicity, a dominated splitting (see section \ref{section:hyperb2}).

\begin{maintheorem}\label{teorema2}
Let $(M,\omega)$ be a $4$-dimensional compact symplectic manifold.
There exists a $C^2$-dense subset $\mathfrak{D}$ of $C^{2}(M,\mathbb{R})$ such that, if $H\in{\mathfrak{D}}$, there exists an invariant decomposition $M=D\cup Z \pmod 0$ satisfying:
\begin{itemize}
\item
$D={\bigcup}_{n\in\mathbb{N}}D_{m_{n}}$, where $D_{m_{n}}$ is a set with $m_{n}$-dominated splitting for the transversal linear Poincar\'{e} flow of $H$, and
\item
the Hamiltonian flow of $H$ has zero Lyapunov exponents for $x\in Z$.
\end{itemize}
\end{maintheorem}

The results above follow closely the strategy applied in~\cite{Be} for volume-preserving flows. 
Besides the decomposition of the manifold into invariant sets for each energy, the main novelty here is the construction of Hamiltonian perturbations. 
Once those are built, we use abstract arguments developed in~\cite{B} and~\cite{Be} to conclude the proofs. 
Nevertheless, for completeness, we will present all the ingredients in the Hamiltonian framework.
At the end of section \ref{section: proof of thm 2} we discuss why in Theorem \ref{teorema2} we are only able to prove the existence of a dense subset instead of residual.

At this point it is interesting to recall a related $C^{2}$-generic dichotomy by Newhouse~\cite{N}. That states the existence of a $C^2$-residual set of all Hamiltonians on a compact symplectic $2d$-manifold, for which an energy surface through any $p\in M$ is Anosov or is in the closure of $1$-elliptical periodic orbits. For another related result, in the topological point of view, we mention a recent theorem by Vivier~\cite{Vi}: any $4$-dimensional Hamiltonian vector field admitting a robustly transitive regular energy surface is Anosov on that surface.

\medskip

In section \ref{section:preliminaries} we introduce the main tools for the proofs of the above theorems (section \ref{section:Proof of the main theorems}).
These are based on Proposition~\ref{main} for which we devote the rest of the paper. The fundamental point is the construction of the perturbations of the Hamiltonian in section~\ref{perturbations}. Finally, we conclude the proof in section~\ref{end} by an abstract construction already contained in~\cite{Be}, which works equally in the present setting.

\end{section}

\begin{section}{Preliminaries}\label{section:preliminaries}

\subsection{Basic notions}
\label{subsection:basic notions}

Let $M$ be a $2d$-dimensional manifold endowed with a symplectic structure, i.e. a closed and nondegenerate 2-form $\omega$. The pair $(M,\omega)$ is called a symplectic manifold which is also a volume manifold by Liouville's theorem. Let $\mu$ be the so-called Lebesgue measure associated to the volume form $\omega^d=\omega\land\dots\land\omega$.

A diffeomorphism $g\colon(M,\omega)\to(N,\omega')$ between two symplectic manifolds is called a symplectomorphism if $g^*\omega'=\omega$. The action of a diffeomorphism on a 2-form is given by the pull-back $(g^*\omega')(X,Y)=\omega'(g_*X,g_*Y)$. Here $X$ and $Y$ are vector fields on $M$ and the push-forward $g_*X=Dg\,X$ is a vector field on $N$. Notice that a symplectomorphism $g\colon M\to M$ preserves the Lebesgue measure $\mu$ since $g^*\omega^d=\omega^d$.

For any smooth Hamiltonian function $H\colon M\to{\mathbb{R}}$ there is a corresponding Hamiltonian vector field $X_H\colon M\to TM$ determined by $\iota_{X_H}\omega=dH$ being exact, where $\iota_v\omega=\omega(v,\cdot)$ is a 1-form. Notice that $H$ is $C^s$ iff $X_H$ is $C^{s-1}$. The Hamiltonian vector field generates the Hamiltonian flow, a smooth 1-parameter group of symplectomorphisms $\varphi^{t}_{H}$ on $M$ satisfying $\frac{d}{dt}{\varphi^{t}_{H}}=X_{H}\circ\varphi^{t}_{H}$ and $\varphi^0_H=\rm{id}$. Since $dH(X_H)=\omega(X_H,X_H)=0$, $X_H$ is tangent to the energy level sets $H^{-1}(e)$. 
In addition, the Hamiltonian flow is globally defined with respect to time because $H|_{\partial M}$ is constant or, equivalently, $X_H$ is tangent to $\partial M$.

If $v\in T_xH^{-1}(e)$, i.e. $dH(v)(x)=\omega(X_H,v)(x)=0$, then its push-forward by $\varphi_H^t$ is again tangent to $H^{-1}(e)$ on $\varphi_H^t(x)$ since 
$$
dH(D\varphi_H^t\,v)(\varphi_H^t(x))=\omega(X_H,D\varphi_H^t\,v)(\varphi_H^t(x))={\varphi_H^t}^*\omega(X_H,v)(x)=0.
$$

We consider also the tangent flow $D\varphi^{t}_{H}:TM\to{TM}$ that satisfies the linear variational equation (the linearized differential equation) 
$$
\frac{d}{dt}{D\varphi^{t}_{H}}=DX_{H}(\varphi^{t}_{H})\, D\varphi^{t}_{H}
$$
with $DX_H\colon M\to TTM$.

We say that $x$ is a \emph{regular} point if $dH(x)\not=0$ ($x$ is not critical). We denote the set of regular points by $\mathcal{R}(H)$ and the set of critical points by $\Sing(H)$.
We call $H^{-1}(e)$ a regular energy level of $H$ if $H^{-1}(e)\cap\Sing(H)=\emptyset$.
A regular energy surface is a connected component of a regular energy level.

Given any regular energy level or surface $\mathcal{E}$, we induce a volume form $\omega_{\mathcal{E}}$ on the $(2d-1)$-dimensional manifold $\mathcal{E}$ in the following way. For each $x\in \EE$,
$$ 
\omega_{\mathcal{E}}(x)=\iota_Y \omega^{d}(x)
\quad\text{on $T_x\EE$}
$$ 
defines a $(2d-1)$ non-degenerate form if $Y\in T_xM$ satisfies $dH(Y)(x)=1$. 
Notice that this definition does not depend on $Y$ (up to normalization) as long as it is transversal to $\EE$ at $x$.
Moreover, $dH(D\varphi_H^t\,Y)(\varphi_H^t(x))=d(H\circ\varphi_H^t) (Y)(x)=1$. Thus, $\omega_{\mathcal{E}}$ is $\varphi_H^t$-invariant, and the measure $\mu_{\mathcal{E}}$ induced by $\omega_{\mathcal{E}}$ is again invariant. In order to obtain finite measures, we need to consider compact energy levels.

On the manifold $M$ we also fix any Riemannian structure which induces a norm $\|\cdot\|$ on the fibers $T_xM$. We will use the standard norm of a bounded linear map $A$ given by $\|A\|=\sup_{\|v\|=1}\|A\,v\|$.

The symplectic structure guarantees by Darboux theorem the existence of an atlas $\{h_j\colon U_j\to\Rr^{2d}\}$ satisfying $h_j^*\omega_0=\omega$ with 
\begin{equation}\label{canonical symplectic form}
\omega_0=\sum_{i=1}^d dy_i\land dy_{d+i}.
\end{equation}
On the other hand, when dealing with volume manifolds $(N,\Omega)$ of dimension $p$, Moser's theorem \cite{Moser} gives an atlas $\{h_j\colon U_j\to\Rr^{p}\}$ such that $h_j^*(dy_1\land\dots\land dy_p)=\Omega$.


\begin{subsection}{Oseledets' theorem for 4-dim Hamiltonian systems}

Unless indicated, for the rest of this paper we fix a $4$-dimensional compact symplectic manifold $(M,\omega)$.
Take $H\in C^2(M,\Rr)$. Since the time-1 map of any tangent flow derived from a Hamiltonian vector field is measure preserving, we obtain a version of Oseledets' theorem~\cite{O} for Hamiltonian systems. 
Given $\mu$-a.e. point $x\in{M}$ we have two possible splittings:
\begin{enumerate}
\item \label{case 1 OS}
$T_{x}M=E_{x}$ with $E_{x}$ $4$-dimensional and
$$
\underset{t\to{\pm{\infty}}}{\lim}\frac{1}{t}\log{\|D\varphi^{t}_{H}(x)\, v\|}=0,
\qquad
v\in E_{x}.
$$
\item \label{case 2 OS}
$T_{x}M=E^{+}_{x}\oplus E^{-}_{x}\oplus{E^{0}_{x}}\oplus{\mathbb{R}X_{H}(x)}$, where $\mathbb{R}X_{H}(x)$ denotes the vector field direction, each one of these subspaces being $1$-dimensional and
\begin{itemize}
\item
$\lim\limits_{t\to\pm\infty}
\frac{1}{t}\log{\|D\varphi^{t}_{H}(x)|_{{E^{0}_{x}}\oplus{\mathbb{R}X_{H}(x)}}\|}=0$;
\item
$\lambda^{+}(H,x)=\lim\limits_{t\to\pm\infty} \frac{1}{t}\log{\|D\varphi^{t}_{H}(x)|_{E^{+}_{x}}\|}>0$;
\item
$\lambda^{-}(H,x)=\lim\limits_{t\to\pm\infty} \frac{1}{t}\log{\|D\varphi^{t}_{H}(x)|_{E^{-}_{x}}\|} = -\lambda^{+}(H,x)$.
\end{itemize}
\end{enumerate}
Moreover,
\begin{equation}\label{angle}
\lim_{t\to{\pm{\infty}}}\frac{1}{t}\log{\det D\varphi^{t}_{H}(x)=\sum_{i\in\{+,-\}}\lambda^{i}(H,x)\,\dim(E^{i}_{x})}=0
\end{equation}
and
\begin{equation}\label{angle2}
\lim_{t\to\pm\infty}\frac1t\log \sin\alpha_t = 0
\end{equation}
where $\alpha_t$ is the angle at time $t$ between any subspaces of the splitting.

The splitting of the tangent bundle is called \emph{Oseledets splitting} and the real numbers $\lambda^{\pm}(H,x)$ are called the \emph{Lyapunov exponents}. In the case \eqref{case 1 OS} we say that the Oseledets splitting is trivial. The full measure set of the \emph{Oseledets points} is denoted by $\mathcal{O}(H)$. 

The vector field direction $\mathbb{R}X_{H}(x)$ is trivially an Oseledets's direction with zero Lyapunov exponent.

\end{subsection}

\begin{subsection}{The transversal linear Poincar\'{e} flow of a Hamiltonian}
\label{subsection:transversal linear Poincare flow}

For each $x\in\mathcal{R}$ (we omit $H$ when there is no ambiguity) take the orthogonal splitting $T_xM=\Rr X_H(x)\oplus N_x$, where $N_x=(\Rr X_H(x))^\perp$ is the normal fiber at $x$. 
Consider the automorphism of vector bundles
\begin{equation}
\begin{split}
D\varphi^{t}_{H}\colon  T_{\mathcal{R}}M & \to T_{\mathcal{R}}M \\
(x,v) & \mapsto  (\varphi^{t}_{H}(x),D\varphi_{H}^{t}(x)\, v).
\end{split}
\end{equation}
Of course that, in general, the subbundle $N_{\mathcal{R}}$ is not $D\varphi_{H}^{t}$-invariant. So we relate to the $D\varphi^{t}_{H}$-invariant quotient space $\widetilde{N}_{\mathcal{R}}=T_{\mathcal{R}}M / \mathbb{R}X_{H}(\mathcal{R})$ with an isomorphism $\phi_{1}\colon N_{\mathcal{R}}\to \widetilde{N}_{\mathcal{R}}$ (which is also an isometry). 
The unique map 
$$
P_{H}^{t}\colon N_{\mathcal{R}}\to N_{\mathcal{R}}
$$ 
such that $\phi_{1}\circ P_{H}^{t}=D\varphi^{t}_{H}\circ\phi_{1}$ is called the \emph{linear Poincar\'{e} flow} for $H$. Denoting by $\Pi_{x}\colon T_xM\to N_x$ the canonical orthogonal projection, the linear map $P^{t}_{H}(x)\colon N_{x}\to N_{\varphi^{t}_{H}(x)}$ is 
$$
P^{t}_{H}(x)\, v=\Pi_{\varphi^{t}_{H}(x)}\circ D\varphi^{t}_{H}(x)\, v.
$$

We now consider 
$$
\NN_x=N_x\cap T_xH^{-1}(e),
$$
where $T_xH^{-1}(e)=\ker dH(x)$ is the tangent space to the energy level set with $e=H(x)$.
Thus, $\NN_\RR$ is invariant under $P^{t}_{H}$.
So we define the map
$$
\Phi_{H}^{t}\colon\mathcal{N}_{\mathcal{R}}\to\mathcal{N}_{\mathcal{R}},
\qquad
\Phi_{H}^{t}=P^{t}_{H}|_{\NN_\RR},
$$
called the \emph{transversal linear Poincar\'{e} flow} for $H$ such that
$$
\Phi^{t}_{H}(x)\colon \mathcal{N}_{x}\to \mathcal{N}_{\varphi^{t}_{H}(x)},
\quad
\Phi^{t}_{H}(x)\, v=\Pi_{\varphi^{t}_{H}(x)}\circ D\varphi^{t}_{H}(x)\, v
$$
is a linear symplectomorphism for the symplectic form induced on $\NN_\RR$ by $\omega$.

If $x\in\mathcal{R}\cap\mathcal{O}$ and $\lambda^+(x)>0$, the Oseledets splitting on $T_{x}M$ induces a $\Phi^{t}_{H}(x)$-invariant splitting $\mathcal{N}_{x}=\mathcal{N}^{+}_{x}\oplus \mathcal{N}^{-}_{x}$ where $\mathcal{N}^{\pm}_{x}=\Pi_{x}(E^{\pm}_{x})$. 

\end{subsection}

\subsection{Lyapunov exponents}

Our next lemma explicits that the dynamics of $D\varphi^{t}_{H}$ and $\Phi_{H}^{t}$ are coherent so that the Lyapunov exponents for both cases are related.

\begin{lemma}\label{equal}
Given $x\in\mathcal{R}\cap\mathcal{O}$, the Lyapunov exponents of the $\Phi_{H}^{t}$-invariant decomposition are equal to the ones of the $D\varphi^{t}_{H}$-invariant decomposition.
\end{lemma}

\begin{proof}
If the Oseledets' splitting is trivial there is nothing to prove. 
Otherwise, let 
$$
n^+=\alpha X_H(x)+v^+\in{\mathcal{N}^{+}_{x}}
$$
with $v^+\in E_x^+$ and $\alpha\in{\mathbb{R}}$. 
We want to study the asymptotic behavior of $\|\Phi_H^t(x)\, n^+\|$. 
 From the following two equalities
\begin{itemize}
\item
$
\Pi_{\varphi_{H}^{t}(x)}D\varphi_{H}^{t}(x)\, X_H(x)=
\Pi_{\varphi_{H}^{t}(x)} X_H\circ\varphi_{H}^{t}(x)=0$,
\item
$\| \Pi_{\varphi_{H}^{t}(x)}D\varphi_{H}^{t}(x)\, v^+\|=
\sin(\theta_t)\|D\varphi_{H}^{t}(x)\, v^+\|$,
\end{itemize}
we get
$$
\lim_{t\to{\pm{\infty}}}\frac{1}{t}\log{\|\Phi^{t}_{H}(x)\, n^{+}\|}=\lim_{t\to{\pm{\infty}}}\frac{1}{t}
\log \left[\sin(\theta_{t})\|{D\varphi_{H}^{t}(x)}\, v^{+}\|\right],
$$
where $\theta_{t}$ is the angle between $X_{H}\circ\varphi_{H}^{t}(x)$ and $E^{+}_{\varphi^{t}_{H}(x)}$.
By \eqref{angle2}, we obtain
\begin{equation*}
\begin{split}
\lim_{t\to{\pm{\infty}}}\frac{1}{t}
\log \left[\sin(\theta_{t})\|{D\varphi_{H}^{t}(x)}\, v^{+}\|\right]
&=
\lim_{t\to{\pm{\infty}}}\frac{1}{t}\log \|D\varphi_{H}^{t}(x)\, v^{+}\| \\
&=
\lambda^{+}(H,x).
\end{split}
\end{equation*}
We proceed analogously for $\mathcal{N}^{-}_{x}$.
\end{proof}

\medskip

Below we state the Oseledets theorem for the transversal linear Poincar\'{e} flow.

\begin{theorem}
Let $H\in C^2(M,\Rr)$. For $\mu$-a.e. $x\in{M}$
there exists the \emph{upper Lyapunov exponent} 
$$
\lambda^{+}(H,x)=\underset{t\to{+\infty}}{\lim}\frac{1}{t}\log\|\Phi_{H}^{t}(x)\| \geq0
$$
and $x\mapsto \lambda^+(H,x)$ is measurable. For $\mu$-a.e. $x$ with $\lambda^{+}(H,x)>0$, there is a splitting $\mathcal{N}_{x}=\mathcal{N}_{x}^{+}\oplus{\mathcal{N}_{x}^{-}}$ which varies measurably with $x$ such that:
$$
\lim_{t\to\pm\infty}\frac{1}{t}\log\|\Phi_{H}^{t}(x)\, v\| =
\begin{cases}
\lambda^{+}(H,x), & {v}\in{\mathcal{N}_{x}^{+}}\setminus\{{0}\} \\
-\lambda^{+}(H,x), & {v}\in{\mathcal{N}_{x}^{-}}\setminus\{{0}\} \\
\pm \lambda^{+}(H,x), & {v}\notin{\mathcal{N}_{x}^{+}} \cup \mathcal{N}_{x}^{-}
\end{cases}
$$
\end{theorem}

\begin{subsection}{Hyperbolic structure}
\label{section:hyperb}

Let $H\in C^2(M,\Rr)$. 
Given any compact and $\varphi^{t}_{H}$-invariant set $\Lambda\subset H^{-1}(e)$, we say that $\Lambda$ is a \emph{hyperbolic set} for $\varphi_H^t$ if there exist $m\in{\mathbb{N}}$ and a $D\varphi_{H}^{t}$-invariant splitting $T_\Lambda H^{-1}(e)=E^{+}\oplus E^{-}\oplus E$ such that for all $x\in\Lambda$ we have:
\begin{itemize}
 \item $\|D\varphi^{m}_{H}(x)|_{E^{-}_{x}}\|\leq\frac{1}{2}$ (uniform contraction),
 \item $\|D\varphi^{-m}_{H}(x)|_{E^{+}_{x}}\|\leq\frac{1}{2}$ (uniform expansion)
\item
and $E$ includes the directions of the vector field and of the gradient of $H$.
\end{itemize}
If $\Lambda$ is a regular energy surface, then $\varphi_H^t|_{\Lambda}$ is said to be {\em Anosov} (for simplicity, we often say that $\Lambda$ is Anosov).
Notice that there are no minimal hyperbolic sets larger than energy level sets.

Similarly, we can define a hyperbolic structure for the transversal linear Poincar\'{e} flow $\Phi_H^t$. We say that $\Lambda$ is hyperbolic for $\Phi_H^t$ on $\Lambda$ if $\Phi_H^t|_\Lambda$ is a hyperbolic vector bundle automorphism.
The next lemma relates the hyperbolicity for $\Phi_H^t$ with the hyperbolicity for $\varphi_H^t$. It is an immediate consequence of a result by Doering~\cite{D} for the linear Poincar\'e flow extended to our Hamiltonian setting and the transversal linear Poincar\'e flow.

\begin{lemma}\label{hyperbolic}
Let $\Lambda$ be an $\varphi_{H}^{t}$-invariant and compact set. Then
$\Lambda$ is hyperbolic for $\varphi_{H}^{t}$ iff $\Lambda$ is hyperbolic for $\Phi_H^t$.
\end{lemma}

We end this section with a well-known result about the measure of hyperbolic sets for $C^2$ (or more general $C^{1+}$) dynamical systems, proved by Bowen~\cite{Bowen}, Bochi-Viana~\cite{BV3} and Bessa~\cite{Be} in several contexts.
Here, following~\cite{Be}, it is stated for Hamiltonian functions, meaning a higher differentiability degree.

\begin{lemma}\label{Bowen}
Let $H\in C^3(M,\Rr)$ and a regular energy surface $\EE$.
If $\Lambda\subset\EE$ is hyperbolic, then $\mu_{\mathcal{E}}(\Lambda)=0$ or $\Lambda=\EE$ (i.e. Anosov).
\end{lemma}

\end{subsection}

\subsection{Dominated splitting}
\label{section:hyperb2}

We now study a weaker form of hyperbolicity. 

\begin{definition}
Let $\Lambda\subset{M}$ be an $\varphi^{t}_{H}$-invariant set and $m\in\Nn$. A splitting of the bundle $\mathcal{N}_\Lambda=\mathcal{N}^{-}_\Lambda\oplus\mathcal{N}^{+}_\Lambda$ is an \emph{$m$-dominated splitting} for the transversal linear Poincar\'{e} flow if it is $\Phi^{t}_{H}$-invariant and continuous such that
\begin{equation}\label{dd}
\frac{\|\Phi^{m}_{H}(x)|\mathcal{N}^{-}_{x}\|}{\|\Phi^{m}_{H}(x)|\mathcal{N}^{+}_{x}\|}\leq{\frac{1}{2}},
\qquad
x\in\Lambda.
\end{equation}
We shall call $\mathcal{N}_\Lambda=\mathcal{N}^{-}_\Lambda\oplus\mathcal{N}^{+}_\Lambda$ a \emph{dominated splitting} if it is $m$-dominated for some $m\in\Nn$.
\end{definition}

If $\Lambda$ has a dominated splitting, then we may extend the splitting to its closure, except to critical points. Moreover, the angle between $\mathcal{N}^{-}$ and $\mathcal{N}^{+}$ is bounded away from zero on $\Lambda$.
Due to our low dimensional assumption, the decomposition is unique. For more details about dominated splitting see~\cite{BDV}.

The above definition of dominated splitting is equivalent to the existence of $C>0$ and $0<\theta<1$ so that
\begin{equation}\label{dd2}
\frac{\|\Phi^{t}_{H}(x)|\mathcal{N}^{-}_{x}\|}{\|\Phi^{t}_{H}(x)|\mathcal{N}^{+}_{x}\|}\leq C\theta^t, 
\qquad
x\in\Lambda,
\quad
t\geq0.
\end{equation}

The proof of the next lemma hints to the fact that the $4$-dimensional setting is crucial in obtaining hyperbolicity from the dominated splitting structure.

\begin{lemma}\label{hyperbolic2}
Let $H\in C^2(M,\Rr)$ and a regular energy surface $\En$.
If $\Lambda\subset \En$ has a dominated splitting for $\Phi_H^t$, then $\overline\Lambda$ is hyperbolic.
\end{lemma}

\begin{proof}
Since $\mathcal{E}$ is compact it is at a fixed distance away from critical points, hence there is $K>1$ such that
$$
\frac1K \leq \|X_H(x)\| \leq K,
\qquad
x\in \EE.
$$ 
On the other hand, because $X_H$ is volume-preserving on the $3$-dimensional submanifold $\EE$, we get
\begin{equation}\label{eq cons X}
\sin(\gamma_0)\,\|X_H(x)\| = \sin (\gamma_t)\,\|X_H\circ\varphi_H^t(x)\|\,
\|\Phi_H^t(x)|_{\NN_x^+}\|\,\|\Phi_H^t(x)|_{\NN_x^-}\|.
\end{equation}
Here $\gamma_t$ is the angle between the subspaces $\NN^-$ and $\NN^+$ at $\varphi_H^t(x)$, which is bounded from below by some $\beta>0$ for any $x\in\overline\Lambda$.
We can now rewrite \eqref{eq cons X} as
\begin{equation*}
\begin{split}
\|\Phi_H^t(x)|_{\NN_x^-}\|^2
& =
\frac{\sin(\gamma_0)}{\sin(\gamma_t)} 
\frac{\|X_H(x)\|}{\|X_H\circ\varphi_H^t(x)\|}
\frac{\|\Phi_H^t(x)|_{\NN_x^-}\|}{\|\Phi_H^t(x)|_{\NN_x^+}\|} \\
&\leq
K^2 \frac{\sin(\gamma_0)}{\sin(\beta)} C\theta^t,
\end{split}
\end{equation*}
where we also have used \eqref{dd2}. Thus we have uniform contraction on $\NN_x^-$.

The above procedure can be adapted for $\NN_x^+$ to find uniform expansion, hence $\overline\Lambda$ is hyperbolic for $\Phi_H^t$. Lemma~\ref{hyperbolic} concludes the proof.
\end{proof}

Combining Lemmas \ref{Bowen} and \ref{hyperbolic2} we get the following.

\begin{proposition}\label{Bowen2}
Let $H\in C^3(M,\Rr)$ and a regular energy surface $\En$.
If $\Lambda\subset\EE$ has a dominated splitting for $\Phi_H^t$, then $\mu_{\En}(\Lambda)=0$ or $\En$ is Anosov.
\end{proposition}

In particular, there is a $C^2$-dense set of $C^2$-Hamiltonians for which the above holds.

\begin{remark}
It is an open problem to decide whether for every $H\in C^3(M,\Rr)$ the following holds: an invariant set $\Lambda$ containing critical points of $H$ and admitting a dominated splitting can only be of zero measure or Anosov.
\end{remark}

\end{section}

\begin{section}{Proof of the main theorems}
\label{section:Proof of the main theorems}

\subsection{Integrated Lyapunov exponent}

Let $H\in C^{2}(M,\mathbb{R})$. We take any measurable $\varphi^{t}_{H}$-invariant subset $\Gamma$ of $M$ and we define the integrated upper Lyapunov exponent over $\Gamma$ by
\begin{equation}\label{def LE}
\LE(H,\Gamma)=\int_{\Gamma}\lambda^{+}(H,x)\,d\mu(x).
\end{equation}

The sequence 
$$
a_n(H)=\int_\Gamma\log\|\Phi^{n}_{H}(x)\|\,d\mu(x)
$$
is subaditive ($a_{n+m}\leq a_n+a_m$), hence $\lim\frac{a_n(H)}n=\inf\frac{a_n(H)}n$.
That is,
\begin{equation}\label{infimum}
\LE(H,\Gamma)=\inf_{n\geq{1}}\frac{1}{n}\int_{\Gamma}\log\|\Phi^{n}_{H}(x)\|\,d\mu(x). 
\end{equation}
Since $H\mapsto\frac{1}{n}\int_{\Gamma}\log\|\Phi^{n}_{H}(x)\|d\mu(x)$ is continuous for each $n$, we conclude that $\LE(\cdot,\Gamma)$ is upper semicontinuous among $C^2$ Hamiltonians having a common invariant set $\Gamma$.


\subsection{Decay of Lyapunov exponent}

For a given Hamiltonian $H\in C^2(M,\Rr)$ and $m\in\Nn$, we define the open set 
$$
\Gamma_{m}(H)=M\setminus D_{m}(H),
$$ 
where $D_{m}(H)$ is the invariant set with $m$-dominated splitting for $\Phi_H^t$. 
This means that $\Gamma_{m}(H)$ is the set of points absent of $m$-dominated splitting.
Furthermore, there exists $\tilde{m}\in\mathbb{N}$ such that for all $m'\geq\tilde{m}$ we have $\Gamma_{m'}(H)\subset\Gamma_m(H)$.
On the other hand, if $H'=H$ on $D_m(H)$, then $\Gamma_m(H')\subset\Gamma_m(H)$. The equivalent relations for $D_m(H)$ are immediate.

The next proposition is fundamental because it allows us to decay the integrated Lyapunov exponent over a full measure subset of $\Gamma_{m}(H)$.

\begin{proposition}\label{main}
Let $H\in C^{s+1}(M,\mathbb{R})$ with $s\geq2$ or $s=\infty$, and $\epsilon,\delta>0$. Then there exists $m\in\Nn$ and $\widetilde{H}\in C^s(M,\Rr)$, $\epsilon$-$C^{2}$-close to $H$, such that $\widetilde H=H$ on $D_m(H)$ and
\begin{equation}\label{dec LE}
\LE(\widetilde H,\Gamma_{m}(H))<\delta.
\end{equation}
\end{proposition}

We assume that $\LE(H,\Gamma_m(H))>0$,
otherwise the claim holds trivially.
We postpone the proof of this proposition to section~\ref{end} and complete the ones of our main results.

\subsection{Proof of Theorem~\ref{teorema22}}

Here we look at the product set 
$$
\mathcal M=M\times C^2(M,\Rr)
$$
endowed with the standard product topology.
Given a point $p$ on the manifold $M$, we denote by $\EE_p(H)$ the energy surface in $H^{-1}(H(p))$ passing through $p$.
The subset 
$$
A=\{(p,H)\in \mathcal M \colon \EE_p(H)\text{ is an Anosov regular energy surface}\}
$$
is open by structural stability of Anosov systems.
Moreover, for each $(p,H)\in A$ there is a tubular neighbourhood of $\EE_p(H)$ in $M$, consisting of regular energy surfaces supporting Anosov flows.

On the complement of the closure of $A$, denoted by 
$$
B=\mathcal M \setminus \overline A,
$$
there is a continuous positive function 
$$
\eta\colon B\to \Rr^+
$$
guaranteeing for $(p,H)\in B$ that $\VV_{p,H}$ is a connected component of 
$$
\{x\in M\colon |H(x)-H(p)|<\eta(p,H)\}
$$
containing $p$ and made entirely of non-Anosov energy surfaces.

Now, for each $k\in\Nn$ write 
$$
A_k=\left\{(p,H)\in B \colon 
\LE(H,\VV_{p,H}) <\frac1k\right\}.
$$
This is an open set because the function in its definition is upper semicontinuous.

\begin{lemma}
$A_k$ is dense in $B$.
\end{lemma}

\begin{proof}
Let $(p,H')\in B$. We want to find an arbitrarly close pair $(p,H)$ in $A_k$.
Notice that we will not need to approximate on the first component, the point on the manifold, but only on the Hamiltonian.

Denote the set of $C^3$ Morse functions on $M$ by $K$. Since Morse functions are $C^2$-dense and have a finite number of critical points, it is sufficient to prove the claim by restricting to $(M\times K)\cap B$.
Moreover, small perturbations of a Hamiltonian in $K$ will have regular energy surfaces through $p$. Therefore, we have a dense subset $D\subset M\times K$ in $B$ such that $\EE_p(H)$ is regular for $(p,H)\in D$ and away from Anosov. This means that in fact we only need to show the claim for $D$.

Let $(p,\widehat H)\in D$, and $\varepsilon>0$ such that $(p, H) \in B$ for any $H$ that is $\varepsilon$-$C^2$-close to $\widehat H$.
Proposition~\ref{main} guarantees that for all $\delta>0$ we can find $H\in C^2(M,\Rr)$ which is $\varepsilon$-$C^2$-close to $\widehat H$ and satisfies
$H=\widehat H$ on $D_m(\widehat H)$ (hence $\Gamma_m(H)\subset\Gamma_m(\widehat H)$) and 
$$
\LE(H,\Gamma_m(\widehat H))<\delta. 
$$

Notice that $\mu(\Gamma_m(H)\cap\VV_{p,H})=\mu(\VV_{p,H})$ for all $m\in\Nn$.
Otherwise, if there was an energy surface $\EE\subset\VV_{p,H}$ and $m\in\Nn$ such that
$\mu_{\EE}(D_m(H)\cap\EE)>0$, by Proposition~\ref{Bowen2} it would be Anosov, thus contradicting that $(p,H)\in B$.

Therefore, since the upper Lyapunov exponent is non-negative,
\begin{equation}
\begin{split}
\LE(H,\VV_{p,H}) 
&=
\LE(H,\Gamma_m(H)\cap\VV_{p,H}) \\
& \leq
\LE(H,\Gamma_m(\widehat H))<\delta.
\end{split}
\end{equation}
The choice $\delta=1/k$ yields $(p,H)\in A_k$.
\end{proof}

 From the above, $A\cup A_k$ is open and dense.
Finally, 
\begin{equation*}
\begin{split}
\mathfrak{A}
&=
\bigcap_{k\in\Nn}(A\cup A_k) =
A\cup \bigcap_{k\in\Nn}A_k \\
&=
A\cup \left\{(p,H)\in B \colon \int_{\VV_{p,H}}\lambda^+(H,x)\,d\mu(x)=0\right\}
\end{split}
\end{equation*}
is residual.
By~\cite{BF} Proposition A.7, we can thus write 
$$
\mathfrak{A}=\bigcup_{H\in\mathfrak{R}}\mathfrak{M}_H\times \{H\},
$$
where $\mathfrak{R}$ is $C^2$-residual in $C^2(M,\Rr)$ and, for each $H\in\mathfrak{R}$, $\mathfrak{M}_H$ is a residual subset of $M$, having the following property:
if $H\in\mathfrak{R}$ and $p\in\mathfrak{M}_H$, then $\EE_p(H)$ is Anosov or 
$$
\int\int\lambda^+d\mu_{\EE}dH=0.
$$
The latter implies that $dH$-a.e. the Lyapunov exponents on each energy surface $\EE$ in $\VV$ are $\mu_{\EE}$-a.e. equal to zero.
Recall that we can split the measure $\mu$ into $\mu_\EE$ on the energy surfaces and $dH$ corresponding to the $1$-form transversal to $\EE$.

Therefore, for a $C^2$-generic $H$, in a neighbourhood of a generic point in $M$ we have the above dichotomy, thus being valid everywhere in the manifold. That completes the proof of Theorem~\ref{teorema22}.

\subsection{Proof of Theorem~\ref{teorema2}}\label{section: proof of thm 2}

It is enough to show that we can arbitrarly $C^2$-approximate any $H\in C^\infty(M,\Rr)$ by $H'\in C^2(M,\Rr)$ satisfying 
$$
LE(H',Z)=0
$$
for some $Z$ to be determined, without domination and whose $\mod0$-complement is dominated.
We use an inductive scheme built on~\eqref{dec LE} and the fact that $\LE(\cdot,\Gamma)$ is an upper semicontinuous function among Hamiltonians having a common invariant set $\Gamma$, to define a convenient sequence $H_n\in C^\infty(M,\Rr)$ with $C^2$-limit $H'$.

Choose a sequence $\epsilon_n\leq\epsilon_0 2^{-n}$ (to be further specified later) for some $\epsilon_0>0$.
By Proposition~\ref{main} we construct the sequence of Hamiltonians $H_n$ in the following way:
\begin{enumerate}
\item
$H_0=H$,
\item
$H_n$ and $H_{n-1}$ are $\epsilon_n$-$C^2$-close,
\item
$H_n=H_{n-1}$ on $D_{m_n}(H_{n-1})$,
\item
$\LE(H_n,\Gamma_{m_n}(H_{n-1}))\leq 2^{-n}$.
\end{enumerate}
That is, each term $H_n$ of the sequence is the perturbation of the previous one $H_{n-1}$ as given by Proposition~\ref{main}.
Then, the $C^2$-limit $H'$ exists and is $\epsilon_n$-$C^2$-close to any $H_n$.

 For each $n$ and an invariant set $\Gamma$ for $H_n$, because $\LE(\cdot,\Gamma)$ is upper semicontinuous, for any $\theta>0$ we can find $\eta_n>0$ such that 
$$
\LE(H_*,\Gamma) \leq (1+\theta)\,\LE(H_n,\Gamma)
$$
as long as $H_n$ and $H_*$ are $\eta_n$-$C^2$-close and have the common invariant set $\Gamma$.

Impose now additionally that $\epsilon_n<\eta_n$. So, for any $n$, 
\begin{equation*}
\begin{split}
\LE(H',\cap_i\Gamma_{m_i}(H_{i-1})) 
&\leq 
\LE(H',\Gamma_{m_n}(H_{n-1})) \\
&\leq
(1+\theta) \LE(H_n,\Gamma_{m_n}(H_{n-1})) \\
&\leq
(1+\theta) 2^{-n}.
\end{split}
\end{equation*}
Therefore, $\LE(H',\cap_{i}\Gamma_{m_i}(H_{i-1}))=0$ and the Lyapunov exponents vanish on
$$
Z=\bigcap\limits_{i\in\Nn}\Gamma_{m_i}(H_{i-1})\pmod 0.
$$

Consider an increasing subsequence $m_{n_k}$.
The complementary set of $\cap_{i}\Gamma_{m_{n_i}}(H_{n_i-1})$ is
$$
D=\bigcup\limits_{i\in\Nn}D_{i},
\quad\text{where}\quad
D_{i}=D_{m_{n_i}}(H_{n_i-1}).
$$
By the inductive scheme above, $D_i\subset D_{i+1}$ and $H'=H_{n_i}$ on $D_i$.
So, $H'$ has an $m_{n_i}$-dominated splitting on $D_{i}$.

Finally, we would like to explain why, unfortunately, the strategy in \cite{BV2} to obtain \emph{residual} instead of \emph{dense} in the hypothesis of Theorem \ref{teorema2}, does not apply in our case. 
We start with a $C^2$ Hamiltonian which is a continuity point of the upper semicontinuous function $H\mapsto \LE(H,M)$ (it is well-known that the set of points of continuity is residual) and define the \emph{jump} (see \cite{BV2} p. 1467) by $\LE(H,\Gamma_{\infty}(H))$, where $\Gamma_{\infty}(H)=\cap_{m}\Gamma_{m}(H)$. 
A continuity point means a zero jump, so that $\lambda^{+}(H,x)=0$ for a.e. $x\in \Gamma_{\infty}(H)$ or else $\Gamma_{\infty}(H)$ has zero measure.
Now, in order to estimate a lower bound for the jump, we will need to perturb the original Hamiltonian $H$ as done in section \ref{perturbations}. 
But Theorem \ref{robinson} becomes useless if
$H$ is $C^2$, because the conjugacy symplectomorphism will only be $C^1$.
Finally, we should note that $C^3(M,\mathbb{R})$ equipped with
the $C^2$-topology is not a Baire space, thus residual sets can be
meaningless.

\end{section}

\begin{section}{Perturbing the Hamiltonian}\label{perturbations}

\begin{subsection}{A symplectic straightening-out lemma}

Here we present an improved version of a lemma by Robinson~\cite{R2} that provides us with symplectic flowbox coordinates useful to perform local perturbations to our original Hamiltonian.

Consider the canonical symplectic form on $\Rr^{2d}$ given by $\omega_0$ as in \eqref{canonical symplectic form}.
The Hamiltonian vector field of any smooth $H\colon\Rr^{2d}\to\Rr$ is then 
$$
X_H=\left[\begin{matrix}0&I\\-I&0\end{matrix}\right]\nabla H,
$$
where $I$ is the $d\times d$ identity matrix.
Let the Hamiltonian function $H_0\colon\Rr^{2d}\to\Rr$ be given by $y\mapsto y_{d+1}$, so that 
$$
X_{H_0}=\frac{\partial}{\partial y_1}.
$$

\begin{theorem}[Symplectic flowbox coordinates]\label{robinson}
Let $(M^{2d},\omega)$ be a $C^s$ symplectic manifold, a Hamiltonian $H\in C^s(M,\Rr)$, $s\geq2$ or $s=\infty$, and $x\in M$. If $x\in\mathcal{R}(H)$, there exists a neighborhood $U\subset M$ of $x$ and a local $C^{s-1}$-symplectomorphism $g\colon (U,\omega)\to(\Rr^{2d},\omega_0)$ such that $H=H_0\circ g$ on $U$.
\end{theorem}

\begin{proof}
Fix $e=H(x)$.
Choose any $C^s$ function $G\colon M\to\Rr$ such that $G(x)=0$ and 
\begin{equation}\label{transversality}
\omega(X_H,X_G)(x)\not=0.
\end{equation}
This defines a transversal $\Sigma$ to $X_H$ at $x$ in the following way. 
If $U\subset M$ is a small enough neighborhood of $x$ in $M$ ($U$ will always be allowed to remain as small as needed), then
$$
\Sigma=G^{-1}(0) \cap U
$$
is a $C^s$ regular connected submanifold of dimension $2d-1$.
Notice also that \eqref{transversality} holds in $U$.

Locally there is a $C^s$ regular $(2d-2)$-dimensional hypersurface of $H^{-1}(e)$ where $H$ and $G$ are both constant: $\Sigma_e=\Sigma\cap H^{-1}(e)$.
Notice that for $m\in\Sigma_e$
\begin{equation}
\begin{split}
T_m\Sigma_e
&=\{v\in T_mM\colon dH(v)(m)=dG(v)(m)=0\} \\
&=\ker(\iota_{X_H}\omega(m)) \cap \ker(\iota_{X_G}\omega(m)).
\end{split}
\end{equation}
Since $\omega(X_H,X_G)\not=0$, we have $X_G(m),X_H(m)\not\in T_m\Sigma_e$ and
$$
T_mM=T_m\Sigma_e\oplus\Rr X_H(m)\oplus\Rr X_G(m).
$$

Now, consider the closed 2-form $\omega_e=\omega|_{\Sigma_e}$ defined on $T\Sigma_e\times T\Sigma_e$.
To show that $(\Sigma_e,\omega_e)$ is a $C^s$ symplectic manifold it is enough to check that $\omega_e$ is non-degenerate.
So, suppose there is $v\in T_m\Sigma_e$ such that $\omega_e(w,v)=0$ for any $w\in T_m\Sigma_e$. 
As in addition $\omega(X_H,v)(m)=\omega(X_G,v)(m)=0$, $m\in\Sigma_e$, due to the fact that $\omega$ is non-degenerate we have to have $v=0$. Thus, $\omega_e$ is non-degenerate.
So, Darboux's theorem assures us the existence of a local diffeomorphism $h\colon\Sigma_e\to\Rr^{2d-2}$ such that 

\begin{equation}\label{Darboux}
h^*\omega_0'=\omega_e
\quad\text{where}\quad
\omega_0'=
\sum_{i=2}^{d}dy_i\land dy_{d+i}. 
\end{equation}

\begin{figure}[h]
\begin{center}
  \includegraphics[width=10cm,height=6cm]{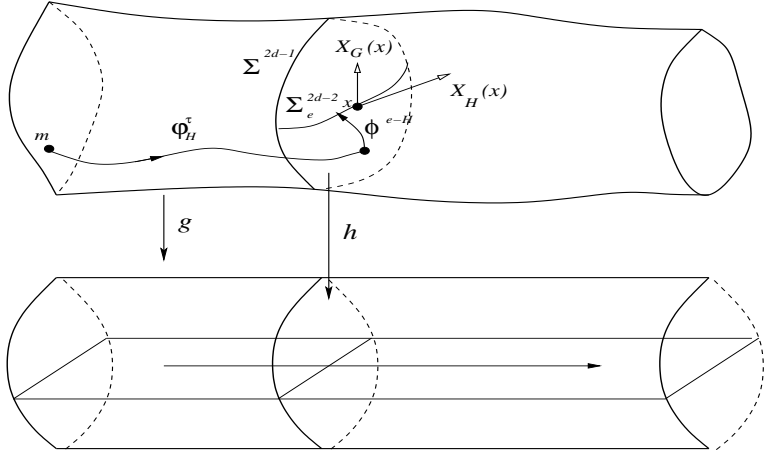}
\caption{The symplectic flowbox.}\label{f2}
\end{center}
\end{figure}

The next step is to extend the above symplectic coordinates from $\Sigma_e$ to $U$. For this purpose we use the parametrization by the flows $\varphi_H^t$ and $\phi^t$ generated by $X_H$ and $Y:=\omega(X_H,X_G)^{-1}X_G$, respectively. The time reparametrization in the definition of $Y$ is necessary to normalize the pull-back of the form as it will become clear later.

The tranversality condition \eqref{transversality} is again used in solving the equation $G\circ \varphi_H^{\tau(m)}(m)=0$, $m\in U$, with respect to a function $\tau\colon U\to\Rr$. That is, we want to find $\tau$ and $U$ such that $\varphi_H^{\tau(m)}(m)\in\Sigma$ for each $m\in U$.
By the implicit function theorem, since $G\circ \varphi_H^{0}(m)=0$ and
$$
\frac{d}{dt}G\circ\varphi_H^t(m)\vert_{t=0}=
dG(X_H)(m)=
\omega(X_G,X_H)(m)\not=0,
$$
there exists $U$ and a unique $\tau\in C^{s-1}(U,\Rr)$ as required.
Moreover, $\phi^t$ preserves the level sets of $G$ as $\LL_YG=\omega(X_G,Y)=0$, and
$$
\LL_YH=\frac{d}{dt}H\circ\phi^t(m)=
\omega(X_H,Y)\circ\phi^t(m)=1.
$$
Thus, $H\circ\phi^t(m)=H(m)+t$ and in particular $H\circ\phi^{e-H(m)}(m)=e$ meaning that $\phi^{e-H(m)}(m)\in H^{-1}(e)$ for $m\in U$.

So, we define the map $g\colon U\to\Rr^{2d}$ given by
$$
g(m)=(-\tau(m),h_1\circ\phi^{e-H(m)}\circ\varphi_H^{\tau(m)}(m),H(m),h_2\circ\phi^{e-H(m)}\circ\varphi_H^{\tau(m)}(m)),
$$
where $h=(h_1,h_2)$ as in \eqref{Darboux} and $h_i\colon\Sigma_e\to\Rr^{d-1}$.
In particular, $H_0\circ g=H$.
It remains to prove that $g$ is a $C^{s-1}$-symplectomorphism.

It follows that $g$ is $C^{s-1}$ and it has a $C^{s-1}$ inverse $g^{-1}\colon g(U)\to U$ given by
$$
g^{-1}(y)=\varphi_H^{y_1}\circ\phi^{y_{d+1}-e}\circ h^{-1}(\widehat y),
$$
where $\widehat y=(y_2,\dots,y_d,y_{d+2},\dots,y_{2d})$.
In addition, for $y\in g(U)$,
\begin{equation}
\begin{split}
g^{-1}_*\, X_{H_0}(y) &=
\dot\varphi_H^{y_1}\circ\phi^{y_{d+1}-e}\circ h^{-1}(\widehat y) \\
&=
X_H\circ\varphi_H^{y_1}\circ\phi^{y_{d+1}-e}\circ h^{-1}(\widehat y) \\
&=
X_H\circ g^{-1}(y).
\end{split}
\end{equation}
Hence, $g_*X_{H}= X_{H_0}$. Similarly, we can show that $g_*Y=\pde{}{y_{d+1}}$ when restricting to $\Sigma$.

Notice that on $g(\Sigma_e)$ we have $g^{-1}_*\pde{}{y_j} =h^{-1}_*\pde{}{y_j}$ for $j\not\in\{1,d+1\}$. Furthermore, taking in addition $k\not\in\{1,d+1\}$,
\begin{equation*}
\begin{split}
({g^{-1}}^*\omega)\left(\pde{}{y_j},\pde{}{y_k}\right) 
&=
({h^{-1}}^*\omega)\left(\pde{}{y_j},\pde{}{y_k}\right)
=\omega_0\left(\pde{}{y_j},\pde{}{y_k}\right), \\
({g^{-1}}^*\omega)\left(\pde{}{y_1},\pde{}{y_{d+1}}\right)
&=
\omega(X_H,Y)\circ g^{-1} =1.
\end{split}
\end{equation*}
Since $Dh^{-1}\pde{}{y_j}\in T\Sigma_e$, and $H$ and $G$ are constant on $\Sigma_e$,
$$
({g^{-1}}^*\omega)\left(\pde{}{y_1},\pde{}{y_j}\right)=
\omega\left(X_H,Dh^{-1}\pde{}{y_j}\right) = 
dH\left(Dh^{-1}\pde{}{y_j}\right)=0
$$
and analogously $({g^{-1}}^*\omega)\left(\pde{}{y_{d+1}},\pde{}{y_j}\right)=0$. Therefore ${g^{-1}}^*\omega$ has to be the canonical 2-form, i.e. $g^*\omega_0=\omega$ on $\Sigma_e$.

Now, we show that $g^*\omega_0=\omega$ also holds on $\Sigma$.
Using Cartan's formula for the Lie derivative $\LL_v=\iota_v d+d\iota_v$ with respect to a vector field $v$ and the identities $df^*=f^*d$ and $f^*\iota_v\omega=\iota_{f_*^{-1}v}f^*\omega$, then
$$
\LL_{Y}g^*\omega_0=g^*d\iota_{\partial/\partial y_{d+1}}\omega_0=g^*d^2(-y_1)=0.
$$
As we also have $\LL_{X_G}\omega=0$ and $\LL_{Y}\omega=0$, the forms $g^*\omega_0$ and $\omega$ are constant and coincide along the flow of $Y$ passing through $\Sigma_e$, i.e. on $\Sigma$.

In order to see that we can have $g^*\omega_0=\omega$ on all of $U$, we compute
$$
\LL_{X_H}g^*\omega_0=d \iota_{X_{H_0}}\omega_0 = d(dH_0)=0.
$$
Recall that $\LL_{X_H}\omega=0$. So, $g^*\omega_0=\omega$ along the flow of $X_H$ through $\Sigma$, thus on all $U$. This concludes the proof that $g$ is a symplectomorphism.
\end{proof}

\subsection{Hamiltonian local perturbation}

In the next lemma we introduce the main tool to perturb $2d=4$-dimensional Hamiltonians. We will then be able to perturb the transversal linear Poincar\'{e} flow in order to rotate its action by a small angle. As we shall see later, that is all we need to interchange $\mathcal{N}^{+}$ with $\mathcal{N}^{-}$ using the lack of dominance.

For functions on $\Rr^4$ consider the $C^k$-norm, with $k\geq0$ integer,
$$
\|f\|_{C^k}= \sup_{y} \max_{0\leq|\sigma|\leq k}
\left |\frac{\partial^{|\sigma|} f(y)}{\partial^{\sigma_1}y_1\dots\partial^{\sigma_4}y_4}\right|,
$$
where $\sigma=(\sigma_1,\dots,\sigma_4)\in \Nn_0^4$ with $|\sigma|=\sum_i \sigma_i$.
Define the ``tube''
$$
V_{a,b,c}=\{(y_1,y_2,y_3,y_4)\in\Rr^4\colon a<y_1<b, \sqrt{y_2^2+y_4^2} < c, |y_3|<c \}.
$$
Moreover, take the $2$-dim plane
$
\Sigma_{0}=\{(0,y_2,0,y_4)\in\Rr^4\}
$
and the orthogonal projection $\pi_{0}\colon\Rr^4\to\Sigma_{0}$.
Notice that the transversal linear Poincar\'e flow of $H_0(y)=y_3$ on $\Sigma_0$ is given by $\Phi_{H_0}^t(y_2,y_4)=\pi_0$.

In the following we fix a universal $0<\varrho<1$.

\begin{lemma}\label{basic}
Given $0<\nu<1$ and $\epsilon>0$, there exists $\alpha_0>0$ such that, for every $0<r<1$ and $0<\alpha\leq\alpha_0$, we can find $H\in C^\infty(\Rr^4,\Rr)$ satisfying
\begin{itemize}
\item
$H=H_0$ outside $V_{0,\varrho,r}$,
\item
$\|H-H_0\|_{C^2}<\epsilon$,
\item
$DX_H(y)=0$ for $y\in\{0,\varrho\}\times\Rr^3$ and
\item
$\Phi_H^1(0,y_2,0,y_4):=(\pi_0 D\varphi_H^1)(0,y_2,0,y_4)=R_{\alpha}$ on $\Sigma_0$ with $\sqrt{y_2^2+y_4^2}<r\nu$, where 
$$
R_{\alpha}=
\begin{bmatrix}
0&0&0&0\\
0&\cos\alpha &0& -\sin\alpha \\
0&0&0&0\\
0&\sin\alpha &0& \cos\alpha
\end{bmatrix}.
$$
\end{itemize}
\end{lemma}

\begin{proof}
Consider the Hamiltonian flow $\varphi_{H_0}^t(y)=(y_1+t,y_2,y_3,y_4)$.
We want to $\epsilon$-$C^2$-perturb $H_0$ to get a Hamiltonian flow that rotates on the $(y_2,y_4)$-plane while the orbit is inside $V_{\xi,\xi',r\nu}$ for some fixed universal constants $0<\xi<\xi'<\varrho<1$. Outside the slightly larger tube $V_{0,\varrho,r}$ we impose no perturbation. 

In order to construct a $C^\infty$ perturbation on those terms, we need to consider three bump functions. 
It is possible to find $C^\infty$ maps $\ell\colon\Rr\to\Rr$ along the time direction and $\widetilde\ell\colon\Rr\to\Rr$ on the $y_3$-direction satisfying 
$$
\ell(y_1)=
\begin{cases}
\ell_0, & y_1 \in [\xi,\xi'] \\
0, & y_1 \not\in (0,\varrho),
\end{cases}
\quad\quad
\widetilde\ell(y_3)=
\begin{cases}
1, & |y_3|\leq r\nu \\
0, & |y_3|\geq r,
\end{cases}
$$
$\ell_0>0$, $\int_0^1\ell=1$, 
the norms $\|\ell\|_{C^0},\|\ell'\|_{C^0},\|\ell''\|_{C^0}$, $\|\widetilde\ell\|_{C^0}$ all bounded from above by a constant (recall that $\xi$, $\xi'$ and $\varrho$ are seen as universal),
$\|\widetilde\ell'\|_{C^0}\leq \frac2{(1-\nu) r}$ and
$\|\widetilde\ell''\|_{C^0}\leq \frac {4}{[(1-\nu) r]^2}$.
Similarly, get a $C^\infty$ map $\phi\colon\Rr^+_0\to\Rr$ for the plane $(y_2,y_4)$ such that
$$
\phi(\rho)=
\begin{cases}
\frac{\rho^2}2, & \rho \leq r\nu  \\
0, & \rho \geq r,
\end{cases}
$$
$\|\phi\|_{C^0}\leq (r\nu)^2$, $\|\phi'\|_{C^0}\leq \frac{2r\nu^2}{1-\nu}$ and $\|\phi''\|_{C^0}\leq \left(\frac{2\nu}{1-\nu}\right)^2$.

Now, we construct the perturbed Hamiltonian
\begin{equation}\label{pert H}
H(y)=H_0(y) - \alpha \ell(y_1)\, \widetilde\ell(y_3)\, \phi(\rho),
\end{equation}
where $\rho=\sqrt{y_2^2+y_4^2}$. Clearly, it is equal to $H_0$ outside $V_{0,\varrho,r}$.
Hence, for $y\in V_{0,1,r\nu}$,
\begin{equation}\label{nabla H}
\nabla H(y)=\left(
-\alpha\ell'(y_1)\phi(\rho),
-\alpha\, y_2\ell'(y_1),
1,
-\alpha\, y_4\ell'(y_1)
\right).
\end{equation}
So, on this domain, $X_H$ generates the flow
\begin{equation}
\begin{split}
\varphi_H^t(y)= & \left(
y_1+t, 
\rho \cos\left(\theta+\alpha\,\int_0^t\ell(y_1+s)ds \right), \right.\\
& 
y_3+\alpha\,\phi(\rho)\,[\ell(y_1+t)-\ell(y_1)], \\
&\left.
\rho \sin\left(\theta+\alpha\,\int_0^t\ell(y_1+s)ds\right)
\right),
\end{split}
\end{equation}
where $\theta=\arctan(y_4/y_2)$.
Notice that $\frac d{dt}\rho^2=0$ so that $\rho$ is $\varphi_H^t$-invariant. That is, on the $(y_2,y_4)$-plane the motion consists of a rotation. 
Furthermore, by fixing $y_3=0$,
$|y_3(t)|\leq \alpha \,|\frac{\rho^2}2|\,|\ell(t)| \leq r\nu$
if $\alpha\leq 2(\|\ell\|_{C^0}r\nu)^{-1}$.
Now, if $\rho<r\nu$,
$$
\varphi_H^1(0,y_2,0,y_4)=(1,\rho\cos(\theta+\alpha),0,\rho\sin(\theta+\alpha))
$$
and $(\pi_0D\varphi_H^1)(0,y_2,0,y_4)\,v=R_\alpha \, v$, $v\in \Sigma_0$.

Finally, we need to estimate the $C^2$-norm of the perturbation.
 From \eqref{pert H} and \eqref{nabla H} we get
\begin{equation}
\|H-H_0\|_{C^1}  \ll \alpha r \nu (1-\nu)^{-1},
\end{equation}
where we are using the notation $A\ll B$ to mean that there is a constant $C>0$ such that $A\leq CB$.
The second order derivatives are
\begin{equation}
\begin{split}
\frac{\partial^2H}{\partial y_1^2} &=-\alpha \ell''(y_1)\widetilde\ell(y_3)\phi(\rho) \\
\frac{\partial^2H}{\partial y_1\partial y_3} &=-\alpha \ell'(y_1)\widetilde\ell'(y_3)\phi(\rho) \\
\frac{\partial^2H}{\partial y_2\partial y_4}& =-\alpha \frac{y_2y_4}{\rho^2} \ell(y_1) \widetilde\ell(y_3) \left(\phi''(\rho)-\phi'(\rho)\rho^{-1}\right) \\
\frac{\partial^2H}{\partial y_1\partial y_j} &=-\alpha y_j \rho^{-1}\ell'(y_1)\widetilde\ell(y_3)\phi'(\rho) \\
\frac{\partial^2H}{\partial y_j^2} &=-\alpha \ell(y_1)\widetilde\ell(y_3)\left[\phi''(\rho)y_j^2\rho^{-2}+\phi'(\rho)\rho^{-1}-\phi'(\rho)y_j^2\rho^{-3}\right]\\
\frac{\partial^2H}{\partial y_3\partial y_j} &=-\alpha \ell(y_1)\widetilde\ell'(y_3)\phi'(\rho)y_j\rho^{-1} \\
\frac{\partial^2H}{\partial y_3^2} &=-\alpha \ell(y_1)\widetilde\ell''(y_3)\phi(\rho)
\end{split}
\end{equation}
where $j=2,4$. So, $DX_H=0$ if $y_1\leq 0$ or $y_1\geq\varrho$, and
\begin{equation}
\begin{split}
\|D^2(H-H_0)\|_{C^0} & \ll  \alpha (1-\nu)^{-2}.
\end{split}
\end{equation}
Hence, there is $\alpha_0\ll \epsilon (1-\nu)^ 2$ such that $\|H-H_0\|_{C^2}<\epsilon$ for all $0<\alpha\leq\alpha_0$.
\end{proof}

\begin{remark}
It is not possible to find $\alpha$ as above if we require $C^3$-closeness. This can easily be seen in the proof by computing the third order derivatives. E.g. $\frac{\partial^3}{\partial y_2^3}H$ contains the term $\alpha \ell(y_1)\widetilde\ell(y_3)y_2^3 \rho^{-3} \phi'''(\rho)$ that can not be controlled by a bound of smaller order than $\alpha r^{-1}$.
\end{remark}

\end{subsection}


\begin{subsection}{Realizing Hamiltonian systems}\label{realizable}

In this section we define the central objects for the proof of Proposition~\ref{main}, the achievable or \emph{realizable} linear flows. These will be constructed by perturbations of $\Phi_H^t$.
We start with a point $x\in\mathcal{O}(H)$ with lack of hyperbolic behavior and mix the directions $\mathcal{N}_{x}^{+}$ and $\mathcal{N}_{x}^{-}$ to cause the decay of the upper Lyapunov exponent. 
In fact we are interested in ``a lot'' of points (related to the Lebesgue measure on transversal sections). Therefore, we perturb the Hamiltonian to make sure that ``many'' points $y$ near $x$ have $\Phi_H^t(y)$ close to $\Phi_H^t(x)$. For this reason we must be very careful in our procedure. 

Consider a Darboux atlas $\{h_{j}\colon U_{j}\to \mathbb{R}^{4}\}_{j\in\{1,...,\ell\}}$.
For each $x\in\mathcal{R}(H)$ choose $j$ such that $x\in U_{j}$,
and take the $3$-dimensional normal section $\mathfrak{N}_{x}$ to the flow. 
In the sequel we abuse notation to write $\mathfrak{N}_{x}$ for $h_j(\mathfrak{N}_{x}\cap U_j)$, so that we work in $\Rr^{4}$ instead of $M$. 
Furthermore, denote by $B(x,r)=\{(u,v,w)\in \Rr^3\colon\sqrt{u^2+v^2}<r,|w|<r\}$ the open ball in $\mathfrak{N}_{x}$ about $x$ with small enough radius $r$. 
We estimate the distance between linear maps on tangent fibers at different base points by using the atlas and translating the objects to the origin in $\Rr^4$.
That is, $\|A_1^t-A_2^t\|$ for linear flows $A_i^t\colon T_{x_i}M\to T_{\varphi_H^t(x_i)}M$, is given by 
$$
\|Dh_{j_{1,t}}(\varphi_H^t(x_1))A_1^t(Dh_{j_{1,0}}(x_1))^{-1}-Dh_{j_{2,t}}(\varphi_H^t(x_2))A_2^t(Dh_{j_{2,0}}(x_2))^{-1}\|,
$$
where $j_{i,t}$ is the indice of the chart corresponding to $\varphi_H^t(x_i)$.

Consider the standard Poincar\'{e} map 
$$
\mathcal{P}_{H}^{t}(x)\colon U \to {\mathfrak{N}_{\varphi_{H}^{t}(x)}},
$$
where $U\subset \mathfrak{N}_{x}$ is chosen sufficiently small. Given $T>0$, the self-disjoint set
$$
\mathcal{F}_{H}^{T}(x,U)=\left\{\mathcal{P}_{H}^{t}(x)\,y\in M\colon y\in{U},t\in[0,T]\right\},
$$
is called a \emph{$T$-length flowbox} at $x$ associated to the Hamiltonian $H$.

There is a natural way to define a measure $\overline{\mu}$ in the transversal sections by considering the invariant volume form $\iota_{X_H}\omega^d$.
We easily obtain an estimate on the time evolution of the measure of transversal sets:
for $\nu,t>0$ there is $r>0$ such that for any measurable $A\subset B(x,r)$ we have
\begin{equation}\label{time ev trans measure}
\left|\muo(A) - \alpha(t)\, \muo(\mathcal P_H^t(x)\,A) 
\right| <\nu,
\end{equation}
where 
$$
\alpha(t)=\frac{\|X_H(\varphi_H^t(x))\|}{\|X_H(x)\|}.
$$

\begin{definition}\label{rlf}
Fix a Hamiltonian $H\in C^{s+1}(M,\Rr)$, $s\geq2$ or $s=\infty$, $T,\epsilon>0$, $0<\kappa<1$ and a non-periodic point $x\in M$ (or with period larger than $T$). The flow $L$ of symplectic linear maps:
$$
L^t(x)\colon \NN_x\to\NN_{\varphi_H^t(x)},
\quad
0\leq t\leq T, 
$$
is {\em $(\epsilon,\kappa)$-realizable of length $T$ at $x$} if the following holds:

For $\gamma>0$ there is  ${r}>0$ such that for any open set $U\subset{B(x,r)}\subset \mathfrak{N}_{x}$ we can find
\begin{enumerate}
\item \label{rlf 0}
$K\subset{U}$ with $\muo(U\setminus K)\leq \kappa\, \muo(U)$, and 
\item \label{rlf 0a}
$\widetilde H\in C^s(M,\Rr)$ $\epsilon$-$C^{2}$-close to $H$, verifying
\begin{enumerate}
\item \label{rlf 1}
$H=\widetilde H$ outside $\mathcal{F}_{H}^{T}(x,U)$,
\item \label{rlf 2}
$DX_{H}(y)=DX_{\widetilde{H}}(y)$ for $y\in U\cup\mathcal{P}_{H}^{T}(x)\,U$, and
\item \label{rlf 3}
$\|\Phi^{T}_{\widetilde{H}}(y)-L^T(x)\|<\gamma$ for all $y\in{K}$.
\end{enumerate}
\end{enumerate}
\end{definition}

Let us add a few words about this definition: \eqref{rlf 1} and \eqref{rlf 2} guarantee that the support of the perturbation is restricted to the flowbox and it $C^1$ ``glues'' to its complement; \eqref{rlf 3} says that a large percentage of points (given numerically by \eqref{rlf 0}) have the transversal linear Poincar\'{e} flow of $\widetilde H$ (as in \eqref{rlf 0a}) very close to the abstract linear action of the central point $x$ along the orbit.
Notice that the realizability is with respect to the $C^{2}$ topology.

\begin{remark}\label{vitali}
Using Vitali covering arguments we may replace any open set $U$ of Definition {\rm \ref{rlf}} by open balls. That turns out to be very useful because the basic perturbation Lemma {\rm\ref{basic}} works for balls.
\end{remark}

It is an immediate consequence of the definition that the transversal linear Poincar\'{e} flow of $H$ is itself a realizable linear flow.
In addition, the concatenation of two realizable linear flows is still a realizable linear flow as it is shown in the following lemma.

\begin{lemma}\label{trivial}
Let $H\in C^{s+1}(M,\Rr)$, $s\geq2$ or $s=\infty$, and $x\in{M}$ non-periodic.
If $L_1$ is $(\epsilon,\kappa_1)$-realizable of length $T_1$ at $x$ and
$L_2$ is $(\epsilon,\kappa_2)$-realizable of length $T_2$ at $\varphi_{H}^{T_1}(x)$ so that $\kappa=\kappa_1+\kappa_2<1$, then the concatenated linear flow 
$$
L^t(x)=
\begin{cases}
L_1^{t}(x), & 0\leq t\leq T_1 \\
L_2^{t-T_1}(\varphi_H^{T_1}(x))\, L_1^{T_1}(x), & T_1<t\leq T_1+T_2 
\end{cases}
$$
is $(\epsilon,\kappa)$-realizable of length $T_1+T_2$ at $x$.
\end{lemma}

\begin{remark}\label{remark concat}
Notice that concatenation of realizable flows worsens $\kappa$.
\end{remark}

\begin{proof}
For $\gamma>0$, take $r_1, r_2,K_1,K_2,\widetilde H_1,\widetilde H_2$ the obvious variables in the definition for $L_1$ and $L_2$.
We want to find the corresponding ones $r,K,\widetilde H$ for $L$ satisfying the properties of realizable flows. Let $x_2=\varphi_H^{T_1}(x)$.

\begin{itemize}
\item
First, choose $r\leq r_1$ such that 
$$
U_2:=\mathcal P_H^{T_1}(x)\, U\subset B(x_2,r_2)
$$
with $U=B(x,r)$.
\item
Now, we construct $\widetilde H$ as
$$
\widetilde H =
\begin{cases}
\widetilde H_1 &\text{on } \mathcal F_H^{T_1}(x,U) \\
\widetilde H_2 &\text{on } \mathcal F_H^{T_2}(x_2,U_2) \\
H & \text{otherwise.}
\end{cases}
$$
Notice that $\mathcal F_H^{T_1+T_2}(x,U) = \mathcal F_H^{T_1}(x,U) \cup \mathcal F_H^{T_2}(x_2,U_2)$.
\item
Consider $K=K_1\cap \mathcal P_H^{-T_1}(x)\,(K_2 \cap U_2)$.
Hence,
\begin{equation*}
\begin{split}
\muo(U\setminus K) & \leq 
\muo(U\setminus K_1) +\muo(U\setminus \mathcal P_H^{-T_1}(x)\,(K_2 \cap U_2)) \\
&\leq 
(\kappa_1+1)\,\muo(U) -\muo(\mathcal P_H^{-T_1}(x)\,(K_2 \cap U_2)).
\end{split}
\end{equation*}
Now, by \eqref{time ev trans measure} applied to $A=\mathcal P_H^{-T_1}(x)\,(K_2 \cap U_2)$ we know that
\begin{equation*}
\begin{split}
\muo(\mathcal P_H^{-T_1}(x)\,(K_2 \cap U_2)) & \geq 
\alpha(T_1)\,\muo(K_2\cap U_2) \\
& =
\alpha(T_1)\,[\muo(U_2)-\muo(U_2\setminus K_2)] \\
& \geq
\alpha(T_1)(1-\kappa_2)\,\muo(U_2).
\end{split}
\end{equation*}
On the other hand, using \eqref{time ev trans measure} for $A=U$, $\muo(U_2)\geq \alpha(T_1)^{-1}\muo(U)$.
Combining all the above estimates we get
$$
\muo(U\setminus K) \leq (\kappa_1+\kappa_2)\,\muo(U).
$$
\item
The choice of $\widetilde H$ yields that $DX_H=DX_{\widetilde H}$ on $U$ because that is true for $\widetilde H_1$. The same on $\mathcal P_H^{T_1+T_2}(x)\,U$ related to $\widetilde H_2$.
\item
In order to check that $\widetilde H$ is $C^s$ it is enough to look at $U_2$. That follows from the same reason as the previous item.
\item
Finally, there is $C>0$ verifying for $y\in K$ and writing $y_2=\mathcal P_H^{T_1}(x)\,y$,
\begin{equation*}
\begin{split}
\|\Phi_{\widetilde H}^{T_1+T_2}(y)-L^{T_1+T_2}(x)\|
\leq &
\|\Phi_{\widetilde H}^{T_2}(y_2)[\Phi_{\widetilde H}^{T_1}(y)-L^{T_1}(x)]\| \\
&
+\|[\Phi_{\widetilde H}^{T_2}(y_2)- L^{T_2}(x_2)]L^{T_1}(x)\| \\
< &
C\gamma.
\end{split}
\end{equation*}
\end{itemize}
\end{proof}

The next lemma is the basic mechanism to perform perturbations in time length $1$, for which we use Lemma~\ref{basic} to realize the map $\Phi_{H}^{t}(x)\circ R_{\alpha}$ (the rotation $R_\alpha$ is defined in a canonical basis of $\mathcal{N}_x$ by the matrix given in Lemma~\ref{basic}).
We will not be needing more than lenght $1$ realizable flows, since we can concatenate them (keeping in mind Remark \ref{remark concat}). Each lenght $1$ piece contributes to rotations by the same angle $\alpha$, independently of $x$, as shown below.

\begin{lemma}\label{rlf1}
Let $H\in C^{s+1}(M,\Rr)$, $s\geq2$ or $s=\infty$, $\epsilon>0$ and $0<\kappa<1$.
Then, there exists $\alpha_0=\alpha_0(H,\epsilon,\kappa)>0$ such that for any non-periodic point $x\in{M}$ (or with period larger than $1$) and $0<\alpha\leq\alpha_0$, the linear flow $\Phi_{H}^{t}(x)\circ R_{\alpha}\colon \NN_x\to\NN_{\varphi_H^t(x)}$ is $(\epsilon,\kappa)$-realizable of length $1$ at $x$.
\end{lemma}

\begin{proof}
Let $\gamma>0$.
We start by choosing $r>0$ sufficiently small such that:
\begin{itemize}
\item
$B(x,r)$ is inside the neighbourhood given by Lemma~\ref{robinson}. Notice that by taking the transversal section $\Sigma=B(x,r)$ in the proof of the lemma, such neighbourhood can be extended along its orbit to an open set $A$ containing $F_H^1(x,r)$, where $F_H^\tau(x,r)=\bigcup_{0\leq t\leq\tau}\varphi_H^t(B(x,r))$.
So, a $C^s$-symplectomorphism $g\colon A\to\Rr^4$ exists satisfying: $g(B(x,r))$ is an orthogonal section to $X_{H_0}$ at $g(x)=0$, $H=H_0\circ g$, and all norms of the derivatives are bounded. 
Moreover, the derivatives of $g$ and $g^{-1}$ are of order $r$-close to the identity tangent map $I$ on local coordinates;
\item
$\mathcal{F}_H^1(x,B(x,r))$ is not self-intersecting;
\item
$F_H^\varrho(x,r)\subset \mathcal{F}_H^1(x,B(x,r))$ (recall that $0<\varrho<1$ is a fixed constant introduced before Lemma~\ref{basic}).
\end{itemize}

Let $U=B(x',r')\subset B(x,r)$, $\widehat\epsilon>0$ and $0<\widehat\kappa<1$. Define $g_{x'}=g-g(x')$ and $\widehat U=g_{x'}(U)$. For $r$ small we find $r_1,r_2>0$ such that $B(0,r_1)\subset \widehat U\subset B(0,r_2)$ and $r_2/r_1=[(1-\widehat\kappa)^{-1/3}+1]/2>1$.
Setting $\nu=[1+(1-\widehat\kappa)^{1/3}]/2<1$, Lemma~\ref{basic} gives us that there is $\alpha_0=\alpha_0(H,\widehat\epsilon,\widehat\kappa)>0$ such that for any $0<\alpha\leq\alpha_0$ and using the radius $r_1$ we have that $\Phi_{H_0}^t(0)\circ R_{\alpha}$ is $(\widehat\epsilon,\widehat\kappa)$-realizable of length-$1$ at the origin. 
Take the obvious variables in the definition $\widehat K\subset\widehat U$ and $\widehat H$ such that $\widehat K =B(0,r_1\nu)$ with $\overline{\mu}(\widehat{K})= \pi (r_1\nu)^3$ and $\|\widehat H-H_0\|_{C^2}<\widehat\epsilon$. Then, $\overline{\mu}(\widehat{K}) \geq (1-\widehat\kappa)\overline{\mu}(\widehat{U})$.

Define $K=g_{x'}(\widehat K)\subset U$ and $\widetilde H=\widehat H\circ g_{x'}$. If $\widehat\epsilon$ and $\widehat\kappa$ are small enough (depending on $\epsilon$, $\kappa$ and the norms of the derivatives of $g$), we get that Definition~\ref{rlf}~(1) is satisfied and 
$$
\|\widetilde H-H\|_{C^2}=
\|(\widehat H-H_0)\circ g_{x'}\|_{C^2}
\ll \widehat\epsilon <\epsilon.
$$
We use the same notation $\ll$ as in the proof of Lemma~\ref{basic}.
By construction it is simple to check that Definition~\ref{rlf}~(2a) and (2b) also hold.

As discussed before, Lemma~\ref{robinson} determines the existence of a neighborhood at each regular point of $M$ and a $C^{s}$-symplectomorphism straightening the flow. 
By compacity of $M$ the derivatives of the symplectomorphism up to the order $s$ are uniformly bounded on small length $1$ flowboxes. For this reason $\alpha_{0}$ given above (depending on $\widehat\epsilon$ and $\widehat\kappa$) was chosen to be independent of $x\in M$.

It remains to check (c) in Definition~\ref{rlf}. This will require further restrictions on $r$, depending on $\gamma$.
By definition, the time-$1$ transversal linear Poincar\'e flow on $\mathcal{N}_y\subset T_yH^{-1}(H(y))$ is
$$
\Phi_{\widetilde H}^1(y) =
\Pi_{\varphi_{\widetilde H}^1(y)} \, Dg^{-1}(\widehat y) \,
D\varphi_{\widehat H}^1(g(y))\, Dg (y)
$$
for $y\in K$, and in $x$ yields
$$
\Phi_H^1(x)\circ R_\alpha =
\Pi_{\varphi_H^1(x)} \, Dg^{-1}(\widehat x) \, Dg(x) \,R_\alpha,
$$
where $\widehat x=\varphi_{H_0}^1\circ g(x)=(1,0,0,0)$ and $\widehat y=\varphi_{\widehat H}^1\circ g(y)$ are of order $r$-close.
Notice that $\|\Pi_{\varphi_{\widetilde H}^1(y)}-\Pi_{\varphi_H^1(x)}\| \ll r$ and
$$
\|Dg^{-1}(\widehat y)-
Dg^{-1}(\widehat x)\|
\ll r.
$$
Therefore, 
$
\|\Phi_{\widetilde H}^1(y)-\Phi_H^1(x)\circ R_\alpha\| \ll
r+ \|\Upsilon\|,
$
where
$$
\Upsilon = \Pi_{\varphi_H^1(x)} Dg^{-1}(\widehat x) \, \left[
D\varphi_{\widehat H}^1(g(y))\, Dg (y) -
Dg(x) \,R_\alpha
\right].
$$
Moreover, $\|Dg^{-1}(\widehat x)-I\|\ll r$. So, 
$$
\|\Upsilon\| \ll r+
\|\Pi_{\varphi_H^1(x)} \left(
R_\alpha\, Dg(y)-Dg(x)\,R_\alpha
\right)\|
$$
where we have also used $\|\Pi_{\varphi_H^1(x)} - \pi_0\|\ll r$ and $\pi_0 D\varphi_{\widehat H}^1(0,y_2,0,y_4)=R_{\alpha}$.
Finally, since 
$$
R_\alpha\, Dg(y)-Dg(x)\,R_\alpha=R_\alpha(Dg(y)-I)+(I-Dg(x))R_\alpha,
$$
we obtain the bound
$$
\|\Phi_{\widetilde H}^1(y)-\Phi_H^1(x)\circ R_\alpha\| \ll r <\gamma
$$
for $r\ll\gamma$ small enough.
\end{proof}

\begin{remark}
A similar result holds true also for $R_{\alpha}\circ{\Phi_{H}^{t}(x)}$ using essentially the same proof.
\end{remark}

\end{subsection}
\end{section}


\begin{section}{Proof of Proposition~\ref{main}}\label{end}

We present here a sketch of how to complete the proof of Proposition~\ref{main}; see~\cite{Be} for full details. 
We would like to highlight the fact that our result does not hold for a $C^2$ Hamiltonian $H$, since the perturbed one $\widetilde H$ has to be one degree of diferenciability less. 
The differentiability loss comes from the symplectomorphism obtained in Theorem~\ref{robinson} that rectifies the flow.

\subsection{Local}

The lemma below states that the absence of dominated splitting is sufficient to interchange the two directions of non-zero Lyapunov exponents along an orbit segment by the means of a realizable flow.

\begin{lemma}\label{exchange}
Let $H\in C^{s+1}(M,\Rr)$, $s\geq2$ or $s=\infty$, $\epsilon>0$ and
$0<\kappa<1$. There exists $m\in{\mathbb{N}}$, such that for
every $x\in{\mathcal{R}(H)\cap\mathcal{O}(H)}$ with a positive Lyapunov exponent and satisfying  
$$
\frac{\|\Phi_{H}^{m}(x)|\mathcal{N}^{-}_{x}\|}{\|\Phi_{H}^{m}(x)|\mathcal{N}^{+}_{x}\|}\geq{\frac{1}{2}},
$$
there exists a $(\epsilon,\kappa)$-realizable linear flow $L$ of length $m$ at $x$ such that 
$$
L^{m}(x)\,\mathcal{N}_{x}^{+}=\mathcal{N}_{\varphi_{H}^{m}(x)}^{-}.
$$
\end{lemma}

\begin{proof}
The proof is the same as for Lemma 3.15 of~\cite{Be} in which the constructions of Lemma~\ref{rlf1} are used, namely the concatenation of rotated Poincar\'e linear maps. 
\end{proof}

Now we aim at locally decaying the upper Lyapunov exponent.

\begin{lemma}\label{smallnorm}
Let $H\in C^{s+1}(M,\Rr)$, $s\geq2$ or $s=\infty$, and $\epsilon,\delta>0$, $0<\kappa<1$. 
There is $T\colon\Gamma_{m}(H)\to{\mathbb{R}}$ measurable, such that for $\mu$-a.e. $x\in{\Gamma_{m}(H)}$ 
and $t\geq{T(x)}$, we can find a $(\epsilon,\kappa)$-realizable linear flow $L$ at $x$ with length $t$ satisfying 
\begin{equation}\label{inequality}
\frac1t\log \|L^{t}(x)\|<\delta. 
\end{equation}
\end{lemma}

\begin{proof}
We follow Lemma 3.18 of~\cite{Be}.
Notice that for $\mu$-a.e. $x\in\Gamma_{m}(H)$ with $\lambda=\lambda^{+}(H,x)>0$ and due to the nice recurrence properties of the function $T$ (see Lemma 3.12 of~\cite{B}) we obtain for every (very large) $t\geq T(x)$ that  
$$
\frac{\|\Phi^{m}_{H}(y)|\mathcal{N}^{-}_{y}\|}{\|\Phi^{m}_{H}(y)|\mathcal{N}^{+}_{y}\|}\geq{\frac{1}{2}}
$$ 
for $y=\varphi_{H}^{s}(x)$ with $s\approx t/2$. 

Now, by Lemma~\ref{exchange} we obtain a $(\epsilon,\kappa)$-realizable linear flow $L_2^t$ such that $L_{2}^{m}\,\mathcal{N}_{y}^{+}=\mathcal{N}_{\varphi_{H}^{m}(y)}^{-}$. 
We consider also the realizable linear flows $L_{1}^t\colon\mathcal{N}_{x}\to\mathcal{N}_{y}$ and $L_{3}^t\colon\mathcal{N}_{\varphi_{H}^{m}(y)}\to\mathcal{N}_{\varphi_{H}^{t}(x)}$ given by $\Phi_H^t$ for $0\leq t\leq s$ and $t\geq m$, respectively. 
Then we use Lemma~\ref{trivial} and concatenate $L_1\to L_2\to L_3$ as $L^t$, which is a $(\epsilon,\kappa)$-realizable linear flow at $x$ with length $t$.

The choice of $t\gg m$ and the exchange of the directions will cause a decay on the norm of $L^t$. Roughly that is:
\begin{itemize}
\item
in $\mathcal{N}^{+}_{x}$ the action of $L_{1}$ is approximately $e^{\lambda t/2}$, 
\item
in $\mathcal{N}^{-}_{\varphi_{H}^{m}(y)}$ the action of  $L_{3}$ is approximately $e^{-\lambda t/2}$ and 
\item
$L_{2}$ exchange these two rates. 
\end{itemize}
Therefore, $\|L^{t}(x)\|<{e^{t\delta}}$.
\end{proof}

\subsection{Global}

Notice that, in Lemma~\ref{smallnorm}, we obtained $\|L^{t}(x)\|<e^{t\delta}$. However, we still need to get an upper estimate of the upper Lyapunov exponent. 
Due to (\ref{infimum}) this can be done without taking limits, say in finite time computations. In other words, we will be using the inequality
\begin{equation}\label{infimum2}
\int_{\Gamma_{m}(H)}\lambda^{+}(\tilde{H},x)d\mu(x)\leq \int_{\Gamma_{m}(H)}\frac{1}{t}\log\|\Phi^{t}_{\tilde{H}}(x)\|d\mu(x),
\end{equation}
which is true for all $t\in\mathbb{R}$.
Therefore, $\delta$ is larger than the upper Lyapunov exponent of at least most of the points near $x$.

To prove Proposition~\ref{main} we turn Lemma \ref{smallnorm} global. This is done by a recurrence argument based in the Kakutani towers techniques entirely described in~\cite{Be} section 3.6. In broad terms the construction goes as follows: 
\begin{itemize}
\item
Take a very large $m\in\mathbb{N}$ from Lemma~\ref{exchange}. 
Then Lemma~\ref{smallnorm} gives us a measurable function $T\colon\Gamma_{m}(H)\to{\mathbb{R}}$ depending on $\kappa$ and $\delta$. Let $\delta^2=\kappa$.
\item
For $x_{1}\in\Gamma_{m}(H)$, the realizability of the flow $L^t(x_{1})$ guarantees that we have a $t$-length flowbox at $x_{1}$ (a tower $\mathcal{T}_{1}$) associated to the perturbed Hamiltonian $\widetilde{H}_{1}$. 
If we take a point in the measurable set $K_{1}$ (cf.~(\ref{rlf 0}) of Definition~\ref{rlf}) contained in the base of the tower, then by~(\ref{rlf 3}) of Definition~\ref{rlf} and Lemma~\ref{smallnorm}, we have $\|\Phi_{\widetilde{H}_{1}}^{t}(y)\|<e^{2\delta t}$ for all $y\in K_{1}$.
\item
Now, for $x_{2},...,x_{j}\in \Gamma_{m}(H)$, where $j\in\mathbb{N}$ is large enough, we define self-disjoint towers $\mathcal{T}_{i}$, $i=1,...,j$, which (almost) cover the set $\Gamma_{m}(H)$ in the measure theoretical sense. 
We take these towers such that their heights are approximately the same, say $h$. 
\item
The $C^s$ Hamiltonian $\widetilde{H}$ is defined by glueing together all perturbations $\widetilde{H}_{i}$, $i=1,...,j$. 
\item
Consider $\mathcal{T}=\cup_i{\mathcal{T}_{i}}$, $U=\cup_{i}U_{i}$ and $K=\cup_{i}K_{i}$. Clearly $K\subset U$. 
Note that for points in $U\setminus K$ we may not have $\|\Phi_{\widetilde{H}_{1}}^{t}(\cdot)\|<e^{2\delta t}$. 
\item
Denote by $\mathcal{T}^{K}$ the subtowers of $\mathcal{T}$ with base $K$ instead of $U$. By~(\ref{rlf 0}) of Definition~\ref{rlf} we obtain that $\overline{\mu}(U\setminus K)\leq\kappa\overline{\mu}(U)$, hence $\mu(\mathcal{T}\setminus\mathcal{T}^{K})<\mu(\mathcal{T})\leq\delta^2$. 

\end{itemize}

We claim that it is sufficient to take $t=h\delta^{-1}$ in \eqref{infimum2}. It follows from~\eqref{inequality} that we only control the iterates that enter the base of $\mathcal{T}^{K}$. 
Since the height of each tower is approximately $h$ the orbits leave $\mathcal{T}^{K}$ at most $\delta^{-1}$ times. 
For each of those times the chance of not re-entering again is less than $\delta^{2}$.
So, the probability of leaving $\mathcal{T}^{K}$ along $t$ iterates is less than $\delta$. 
In conclusion, most of the points in $\Gamma_{m}(H)$ satisfy the inequality \eqref{inequality} and Proposition~\ref{main} is proved.

\end{section}


\subsection*{Acknowledgments} 

We would like to thank Gonzalo Contreras and the anonymous referee for useful comments. MB was supported by Funda\c c\~ao para a Ci\^encia e a Tecnologia, SFRH/BPD/20890/2004. JLD was partially supported by Funda\c c\~ao para a Ci\^encia e a Tecnologia through the Program~FEDER/POCI~2010.


\end{document}